\documentclass[12pt]{amsart}
\usepackage[margin=1.1in]{geometry}
\usepackage{amssymb,amsmath,amsthm,enumitem,graphicx,tikz}
\usepackage[noabbrev,capitalize]{cleveref}
\crefname{equation}{}{}
\usepackage{dsfont, color}
\usepackage{subcaption}
\usepackage[parfill]{parskip}
\setlist[itemize]{parsep=0pt}
\setlist[enumerate]{parsep=0pt}
\usepackage{ tipa }
\usepackage{siunitx}

\makeatletter
\newtheorem*{rep@theorem}{\rep@title}
\newcommand{\newreptheorem}[2]{%
\newenvironment{rep#1}[1]{%
 \def\rep@title{#2 \ref{##1}}%
 \begin{rep@theorem}}%
 {\end{rep@theorem}}}
\makeatother

\newtheorem{theorem}{Theorem}[section] 

\newtheorem{lemma}[theorem]{Lemma} \newreptheorem{theorem}{Theorem}

\theoremstyle{definition}
\newtheorem{definition}[theorem]{Definition}

\newtheorem{question}[theorem]{Question}
\newtheorem{conjecture}[theorem]{Conjecture}

\newtheorem{observation}[theorem]{Observation}

\newcommand{\floor}[1]{\left\lfloor #1 \right\rfloor}

\newcommand{\paren}[1]{\left( #1 \right)}

\newcommand{\set}[1]{\left\{ #1 \right\}}

\DeclareMathOperator*{\argmax}{argmax}
\DeclareMathOperator*{\argmin}{argmin}

\author{Natalya Ter-Saakov}
\author{Emily Zhang}
\address{Massachusetts Institute of Technology, Cambridge, MA 02139, USA}
\email{\{natalyat,eyzhang\}@mit.edu}
\title{Extremal Pattern-Avoiding Words}

\begin{document}

\maketitle

\begin{abstract}
Recently, Grytczuk, Kordulewski, and Niewiadomski defined an \textit{extremal word} over an alphabet $\mathds{A}$ to be a word with the property that inserting any letter from $\mathds{A}$ at any position in the word yields a given pattern. In this paper, we determine the number of extremal $XY_1XY_2X\dots XY_tX$-avoiding words on a $k$-letter alphabet. We also derive a lower bound on the shortest possible length of an extremal square-free word on a $k$-letter alphabet that grows exponentially in $k$. 
\end{abstract}

\section{Introduction}
A \textit{word} is a sequence of letters from a given alphabet. 
A word $U$ is a \textit{factor} of a word $W$ if $W=W_1UW_2$ for some (possibly empty) words $W_1$ and $W_2$.
A \textit{square} is a nonempty word of the form $XX$, and a word is \textit{square-free} if it does not contain a square as a factor. 

Some of the first work on square-free words was by the Norwegian mathematician Axel Thue who showed in the early 1900's that there are arbitrarily long square-free words over a three-letter alphabet \cite{thue1906uber}.
Thue's work \cite{thue1906uber, thue1906uber2} is considered to be the beginning of the field of combinatorics on words \cite{berstel2007origins}.

\begin{definition}
A word $W$ over an alphabet $\mathds{A}$ \textit{realizes} a pattern $P=p_1p_2\dots p_r$, for letters $p_i$, if $W$ can be split into nonempty factors $W_1, \dots, W_r$ such that $W_i=W_j$ if $p_i=p_j$. A word $W$ \textit{contains} a pattern $P$ if there is a factor of $W$ that realizes $P$. Otherwise, $W$ \textit{avoids} $P$.
\end{definition}

\begin{definition}
A pattern $P$ is \textit{unavoidable} on an alphabet $\mathds{A}$ if every sufficiently long word on $\mathds{A}$ contains $P$. Otherwise, $P$ is \textit{avoidable} on $\mathds{A}$. 
\end{definition}

A complete characterization of avoidable patterns was given independently by Bean, Ehrenfeucht, and McNulty \cite{bean1979avoidable} in 1979 and Zimin \cite{zimin1984blocking} in 1984.
\begin{theorem} \cite{bean1979avoidable, zimin1984blocking}\label{thm:zimin}
A pattern is unavoidable if and only if it is contained in a Zimin word, defined recursively as follows: $Z_1 = x_1$ and $Z_{n+1} = Z_{n}x_{n+1}Z_{n}$ for ${n\in \mathds{Z}^+}$.
\end{theorem}

Recently, Grytczuk, Kordulewski, and Niewiadomski \cite{grytczuk2019extremal} introduced an extremal variant of the study of pattern-avoiding words:
\begin{definition}
For a fixed alphabet $\mathds{A}$ and a finite word $W$ over $\mathds{A}$, an \textit{extension} of $W$ is a word $W'=W_1xW_2$ for some (possibly empty) words $W_1$ and $W_2$, where $W=W_1W_2$ and $x\in\mathds{A}$. Given a fixed pattern $P$, a word $W$ is \textit{extremal $P$-avoiding} if $W$ avoids $P$ and every extension of $W$ contains $P$. 
\end{definition}

For example, the word $$abcabacbcabcbabcabacbcabc$$ is a shortest extremal square-free word over the three-letter alphabet $\mathds{A} = \{a,b,c\}$ \cite{grytczuk2019extremal}. In their paper, Grytczuk et al.\ \cite{grytczuk2019extremal} showed that 
there are infinitely many extremal square-free words over a three-letter alphabet. Their method for generating extremal square-free words can be used to generate an infinite extremal square-free word.


In 2020, Mol and Rampersad \cite{mol2000lengths} adapted the ideas of Grytczuk et al.\ to find all integers $n$ for which an extremal square-free ternary word of length $n$ exists. Mol, Rampersad, and Shallit then studied extremal overlap-free binary words in \cite{mol2020extremal}, where an \textit{overlap} is a word of the form $aXaXa$, for any letter $a$ and any (possibly empty) word $X$. In our paper, we examine extremal pattern-avoiding words for other types of patterns, motivated by the following conjecture of Grytczuk et al.:

\begin{conjecture} \cite{grytczuk2019extremal} \label{conjecture: pattern}
For every avoidable pattern $P$, there exists a constant $k(P)$ such that the set of extremal $P$-avoiding words over a $k(P)$-letter alphabet is finite.
\end{conjecture}

We derive an expression for the number of extremal $P$-avoiding words for each pattern $P$ of the form $XY_1XY_2X\dots XY_tX$ where $t\in \mathds{Z}^+$.
We note that \cref{thm:zimin} implies that these patterns are unavoidable.

\begin{theorem}\label{thm: natasha_main}
For $t\in \mathds{Z}^+$, the number of extremal $XY_1XY_2X\dots XY_tX$-avoiding words over a $k$-letter alphabet is $$\frac{(tk)!}{t^k}.$$
\end{theorem}

We further extend the work of Grytczuk et al.\ \cite{grytczuk2019extremal}, who specifically studied ternary extremal square-free words, by studying extremal square-free words on alphabets of larger sizes.
\begin{theorem}\label{thm: emily_main}
For any integer $k\geq 3$, every extremal square-free word on a $k$-letter alphabet has length greater than $$ \paren{\frac{5}{4}}^{\frac{k}{4}}.$$
\end{theorem}
\subsection{Notation}
Throughout this paper, we use the notation $[n]:=\{1,2,\dots,n\}$. 
For an arbitrary set $S$, a totally ordered set $Y$, and a function $g:S\rightarrow Y$, we define $\argmax_{x\in S} g(x)$ to be any 
$x^*\in S$ such that $g(x) \leq g(x^*)$ for all $x\in S$. We define $\argmin_{x\in S} g(x)$ analogously.

\subsection{Outline}
The rest of this paper is structured as follows. In \cref{section: natasha}, we prove \cref{thm: natasha_main}. In \cref{section: almost-squares}, we prove \cref{thm: emily_main}. Finally, we conclude with some directions for future research in \cref{section:future}.

\section{Extremal $XY_1XY_2X\dots XY_tX$-avoiding words} \label{section: natasha}

In this section, we prove our formula for the number of extremal $XY_1XY_2X\dots XY_tX$-avoiding words on a $k$-letter alphabet. We will start by looking at the cases $t=1,2$ and then generalize. 


\begin{observation}
Any extremal $XY_1XY_2X\dots XY_tX$-avoiding word over a fixed alphabet $\mathds{A}$ must use all of the letters in $\mathds{A}$, since inserting an unused letter at the beginning of an extremal $XY_1XY_2X\dots XY_tX$-avoiding word forms an extension that avoids $XY_1XY_2X\dots XY_tX$.
\end{observation}

\begin{lemma}\label{lemma:nat1}
There are $k!$ extremal $XYX$-avoiding words on a $k$-letter alphabet.
\end{lemma}

\begin{proof}
We will start by considering the single-letter alphabet $\mathbb{A}=\{a\}$. The word $aaa$ is the shortest word to contain $XYX$ ($X=a, Y=a)$. So, the word $aa$ is extremal $XYX$-avoiding. Notice that any longer word is not $XYX$-avoiding and any shorter word is not extremal.

Now we consider the two-letter alphabet $\mathbb{A}=\{a,b\}$. All instances of the letter $a$ in an $XYX$-avoiding word must be consecutive. Furthermore, the factor of $a$'s must be of length $2$ for the word to avoid $XYX$ and be extremal ($abb$ is not extremal $XYX$-avoiding because the extension formed by inserting an $a$ after the first letter does contain $XYX$). As such, the only extremal $XYX$-avoiding words on a two-letter alphabet are $aabb$ and $bbaa$.

Having looked at specific cases, we consider the $k$-letter alphabet $\mathbb{A}=\{a_1, a_2, \dots , a_k\}$. Any instances of the letter $a_i$ in an $XYX$-avoiding word must be consecutive. The factor of $a_i$'s must be of length $2$ for the word to avoid $XYX$ and be extremal. So any extremal $XYX$-avoiding word on the $k$-letter alphabet will be of the form $$a_{i_1}a_{i_1}a_{i_2}a_{i_2}\dots a_{i_k}a_{i_k}.$$  Therefore, there are $k!$ extremal $XYX$-avoiding words on a $k$-letter alphabet.
\end{proof}

\begin{lemma} \label{lemma:nat2}
The number of extremal $XY_1XY_2X$-avoiding words on a $k$-letter alphabet is $$\frac{(2k)!}{2^k}.$$
\end{lemma}

\begin{proof}
Let us start by considering the single-letter alphabet $\mathbb{A}=\{a\}$. The word $aaaaa$ is the shortest word to contain $XY_1XY_2X$ ($X=a, Y_1=a, Y_2=a$). So the only extremal $XY_1XY_2X$-avoiding word is $aaaa$.

Now we consider the two-letter alphabet $\mathbb{A}=\{a,b\}$. All instances of the letter $a$ in an $XY_1XY_2X$-avoiding word must either be consecutive or split into two factors separated by a consecutive sequence of the other letter.
If all of the instances of the letter $a$ are consecutive, then the factor of $a$'s must be of length $4$ in order for the word to be extremal. Otherwise, the two factors of $a$'s must both be of length $2$; to see this, consider an extension of the word formed by inserting $a$ into one of the two factors of $a$'s. The extension must contain $XY_1XY_2X$ with $X=a$, so one of the factors of $a$'s must realize $XYX$ after the extension. Thus, both of factors of $a's$ must be extremal $XYX$-avoiding. It follows that there are six extremal $XY_1XY_2X$-avoiding words:
\begin{align*}
    &aaaabbbb &&aabbbbaa 
    &aabbaabb\\
    &bbbbaaaa &&bbaaaabb
    &bbaabbaa.
\end{align*}

Finally we consider the $k$-letter alphabet $\mathbb{A}=\{a_1, a_2, \dots , a_k\}$. All instances of the letter $a_i$ in an $XY_1XY_2X$-avoiding word must either be consecutive in a factor of length $4$ or in two consecutive sequences, each of length $2$. Thus, the extremal $XY_1XY_2X$-avoiding words are in bijection with permutations of the multiset $\{1,1,2,2,\dots ,k,k\}$.
In such a permutation, each number $i$ corresponds to a factor of $a_i$'s of length $2$. For example, the permutation $123321$ corresponds bijectively to the word $$a_1a_1a_2a_2a_3a_3a_3a_3a_2a_2a_1a_1,$$ an extremal $XY_1XY_2X$-avoiding word on a three-letter alphabet. Therefore, the number of extremal $XY_1XY_2X$-avoiding words on a $k$-letter alphabet is 
\begin{align*}
    \frac{(2k)!}{2^k}. &\qedhere
\end{align*}
\end{proof}

Having looked at the cases for $t=1,2$, we will now take the elements and understanding from the lemmas to generalize to the proof of \cref{thm: natasha_main}.

\begin{proof}[Proof of \cref{thm: natasha_main}] 
Let $\mathbb{A}=\{a_1,a_2,\dots,a_k\}$. We claim that the set of extremal $XY_1XY_2X\dots XY_tX$-avoiding words over a $k$-letter alphabet is in bijection with the set of permutations of the multiset $\{1,1,\dots,1,2,2,\dots,2,\dots, k,k,\dots,k\}$ where every element occurs $t$ times. 
In such a permutation, a consecutive substring of $i$'s of length $j$ corresponds to a factor of $a_i$'s of length $2j$ in the extremal $XY_1XY_2X\dots XY_tX$-avoiding word. Since $2j$ is the length of an extremal $XY_1XY_2X\dots XY_jX$-avoiding word on a one-letter alphabet, each factor of $2j$ consecutive $a_i$'s does not realize $XY_1XY_2X\dots XY_jX$, so we know that our word does not contain  $XY_1XY_2X\dots XY_tX$ with $X=a_i$ overall. 
Now we consider extensions of the word created from such a permutation. If the letter $a_i$ is inserted into a factor of $a_i$'s of length $2j$, that factor will realize $XY_1XY_2X\dots XY_jX$, and the whole word will contain $XY_1XY_2X\dots XY_tX$. Inserting $a_i$ anywhere else in the word will also create an extension that contains $XY_1XY_2X\dots XY_tX$, where we take one $a_i$ for every $i$ in the permutation along with the one that was just inserted to be the $X$'s in the pattern. 
By symmetry of the $a_i$'s, the word created from a permutation of the multiset is an extremal $XY_1XY_2X\dots XY_tX$-avoiding word. 

Extending the reasoning from \cref{lemma:nat1} and \cref{lemma:nat2}, we can see that all extremal $XY_1XY_2X\dots XY_tX$-avoiding words must be of the form described above.
Therefore, the number of extremal $XY_1XY_2X\dots XY_tX$-avoiding words on a $k$-letter alphabet is 
\begin{align*}
    \frac{(tk)!}{t^k}. &\qedhere
\end{align*}
\end{proof}

\section{Extremal square-free words on a $k$-letter alphabet} \label{section: almost-squares}
In this section, we prove \cref{thm: emily_main} by introducing and studying a new type of word that we call an \textit{almost-square}:

\begin{definition}
An \textit{almost-square} is a word of the form $WW'$, where $W'$ is either an extension of $W$ or is obtained by deleting one letter from $W$. 
\end{definition}

For example, the words $a$, $aba$, and $ababc$ are almost-squares.
We note that all almost-squares have odd length.

We will use an upper bound on the number of almost-squares in any square-free word to derive the lower bound in \cref{thm: emily_main}.
It is easy to see that the number of almost-squares in any word of length $n$ is $O(n^2)$. We will show that the number of almost-squares in any square-free word of length $n$ is $O(n\log n)$. We begin with a lemma.

\begin{lemma} \label{lemma:almost-square}
Let $WW'$ be any almost-square on an alphabet $\mathds{A}$, and let $V$ be a word on the same alphabet such that $WW'V$ is an almost-square $XX'$. If $4\leq|V|\leq \frac{1}{2}|W|$, then $XX'$ is not square-free. 
\end{lemma}
\begin{proof}
\begin{sloppypar}
Let $WW'V = XX'$, where $WW'$ and $XX'$ are almost-squares and ${4\leq |V| \leq \frac{1}{2}|W|}$.
For the sake of contradiction, we assume that $XX'$ is square-free.
We show in the appendix that we have $|W|<|X|<|WW'|$, and the figure below shows the word $XX'$ broken up in two different ways: as $WW'V$ and $XX'$. 
We define ${\ell := |V| - (|X|-|W|)}$, and we show in the appendix that $0 < \ell <|W| -1$ and $|V|<|X|-1$. The last $\ell$ letters of both $W$ and $W'$ are indicated with blue in the figure below, and the last $|V|$ letters of $X$ and $X'$ are indicated with red. 
\begin{center}
\includegraphics[scale=.65]{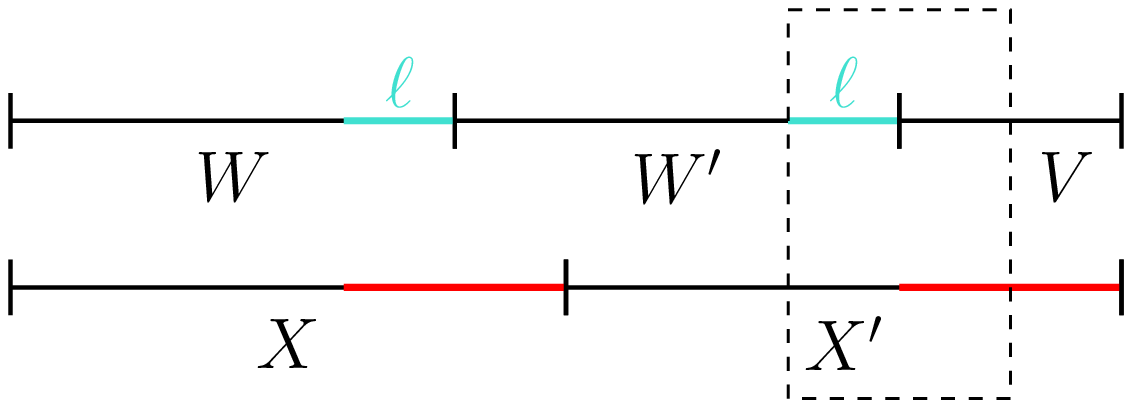}
\end{center}

We consider the following two statements:\\
\noindent Statement 1. The two factors of the word indicated above with blue are identical. \\
Statement 2. The two factors of the word indicated above with red are identical.

If both Statement 1 and Statement 2 hold true, then the first $\ell$ letters of both of the red segments are identical to both of the identical blue segments, so the factor of the word $XX'$ enclosed above in the dashed box is a square, a contradiction. Thus, at least one of the two statements must fail.

\noindent \textbf{Case 1.} Statement 2 fails. \\
Consider the figure shown below. 
The first $|X| - |V|$ letters of $X$ and $X'$ (shown in green) must be identical in this case. 

The first $|W| - \ell - 1$ letters of $W$ and $W'$ are shown in orange. 
Since ${|V|\leq \frac{1}{2}|W|<|W|-1}$, we have that $|W| + (|W|-\ell -1)> |X|$, so the orange part of $W'$ must overlap with $X'$.
The inequality $|V|< |W|$ implies $|X|-|V| > |X| - |W|$, so each factor shown in green is longer than the factor of $XX'$ that is a factor of both $W'$ and $X$. 
Thus, if the two factors shown in orange are identical, then the factor highlighted in yellow is a square of length $2(|X|-|W|)$.

It follows that the last $\ell+1$ letters of $W$ and $W'$ (shown in blue; figure not precise) must be identical.
We can check that we have $|W| + (|W'| - \ell - 1) + 2 \leq |X| + (|X|-|V|)$, so the blue part of $W'$ and the green part of $X'$ overlap by at least two letters. Additionally, the inequality $|V|\leq \frac{1}{2}|W|$ gives us $2(\ell +1)\leq |W|$. Putting all of this together, we see that the factor enclosed in the dashed box is a square of length at most $2(\ell +1)$, so we have reached a contradiction.
\begin{center}
\includegraphics[scale=.65]{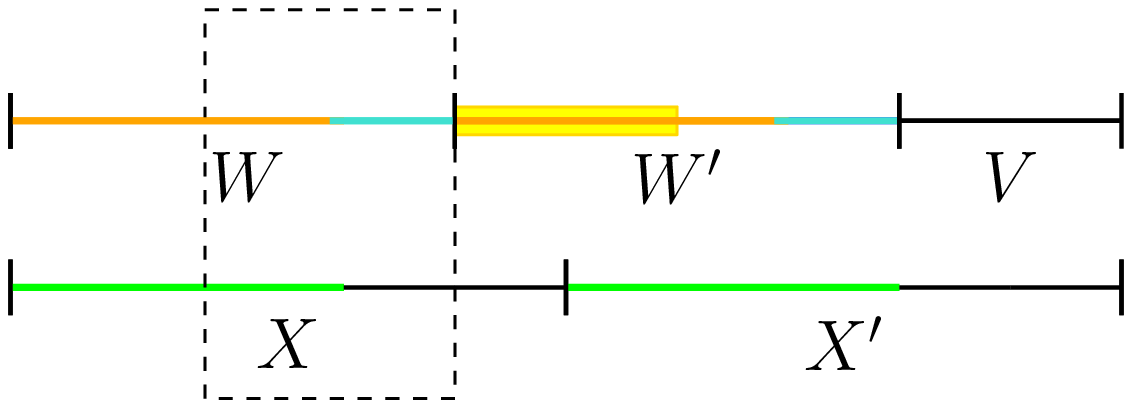}
\end{center}
\noindent \textbf{Case 2.} Statement 1 fails.\\
Consider the figure shown below. 
The first $|W|-\ell$ letters of $W$ and $W'$ (shown in orange) must be identical in this case. The last $|W|$ letters of $X$ and $X'$ are shown in red.
We can check that we have $|W|-\ell > |X| - |W|$, so the orange part of $W$ overlaps with the red part of $X$. Similarly, we can check that we have $|W|+(|W|-\ell)\geq |X| + (|X'| -|W|)$, so the orange part of $W'$ overlaps with or is adjacent to the red part of $X'$.

If the first $|X|-|W|$ letters of $X$ and $X'$ (shown in green) are identical, then the factor highlighted in yellow is a square of length $2(|X| -|W|)$. Thus, the last $|W|$ letters of $X$ and $X'$ (shown in red) must be identical. 
However, this implies that the factor enclosed in the dashed box is a square, so we have reached a contradiction.
\begin{center}
\includegraphics[scale=.65]{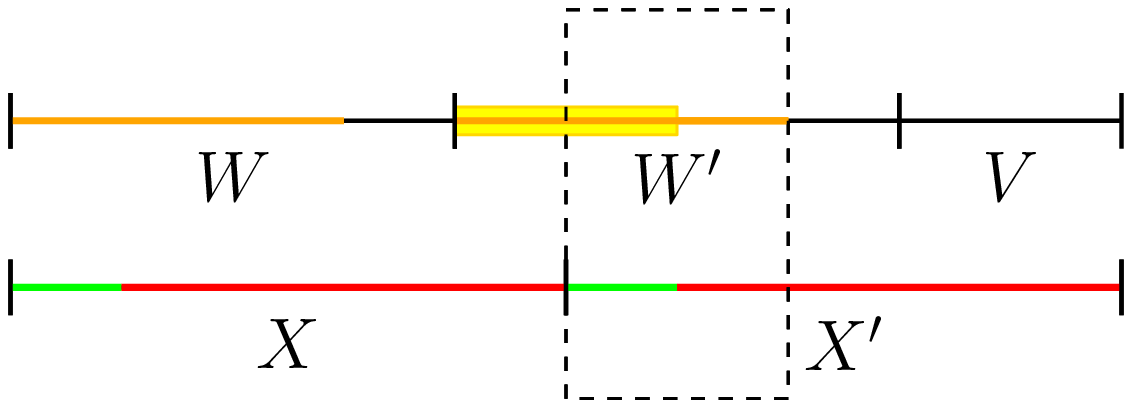}
\end{center}
\end{sloppypar}
\end{proof}

\begin{lemma} \label{lemma: almost-square2}
For any square-free word $Q$ of length $n\geq 2$, the number of almost-squares in $Q$ that start at any given letter of $Q$ is less than $$2\log_{\frac{5}{4}}{n}.$$
\end{lemma}
\begin{proof}
Let $Q$ be any square-free word of length $n\geq 2$. We define $$p := \argmax_{i\in[n]}\paren{\text{number of almost-square factors of }Q\text{ that start at the }i\text{\textsuperscript{th} letter of }Q},$$
and we let $$A = \set{a\mid Q \text{ contains an almost-square of length }a\text{ that starts at the } p\text{\textsuperscript{th} letter of }Q}.$$
We consider the nontrivial case, where $|A|>2$. 
Let $\paren{a_i}_{i\in [|A|]}$ be the strictly increasing sequence of the integers in $A$. Because all almost-squares have odd length, $a_{i+1} - a_i \geq 2$ for all $i\in [|A| - 1]$. We fix some integer $i\in [|A|-2]$ and consider  an almost-square $WW'$ of length $a_i$ that starts at the $p$\textsuperscript{th} letter of $Q$.

Let $XX'= WW'V$ be a factor of $Q$ that starts at the $p$\textsuperscript{th} letter of $Q$ and is an almost-square of length $a_{i+2}$. 
Since $Q$ is square-free and $|V| = a_{i+2}-a_i\geq 4$, \cref{lemma:almost-square} implies that we must have $|V|> \frac{1}{2}|W|$. It follows that $$|XX'| > |WW'| + \frac{1}{2}|W|.$$
Rewriting, we have $$a_{i+2}\geq \frac{5}{4}a_i,$$ which (combined with the fact that $a_2,a_3 \geq 3$) implies 
$$a_{2j}> \paren{\frac{5}{4}}^j\quad\text{and}\quad a_{2j+1}> \paren{\frac{5}{4}}^j\quad\text{for all}\quad 1\leq j\leq \floor{\frac{|A|}{2}}.$$
Thus, $$|A| <  2\log_{\frac{5}{4}}{n},$$ and we are done because the maximum number of almost-squares that start at any fixed letter in $Q$ is $|A|$.
\end{proof}
\cref{lemma: almost-square2} directly implies the following theorem.

\begin{theorem} \label{thm: almost-square}
The number of almost-squares in any square-free word of length $n\geq 2$ is less than $$2n\log_{\frac{5}{4}}{n}.$$
\end{theorem}

We will use the result in \cref{thm: almost-square} to prove \cref{thm: emily_main}. We begin with a lemma.

\begin{lemma} \label{lemma: distinct-extension}
All square-free extensions of a word are distinct.
\end{lemma}
\begin{proof}
\begin{sloppypar}
Let $W = w_1w_2\dots w_n$, with $w_i$ in alphabet $\mathds{A}$. 
For some $\alpha, \beta \in \mathds{A}$ and integers ${1\leq i\leq j\leq n+1}$, let $E_1$ be the extension with letter $\alpha$ inserted as the $i$\textsuperscript{th} letter of $W$, and let $E_2$ be the extension with $\beta$ inserted as the $j$\textsuperscript{th} letter of $W$.
\end{sloppypar}

Suppose, for contradiction, that $E_1$ and $E_2$ are formed by extending the word $W$ in two different ways and that they are identical words. Then, $\alpha = \beta$ and $i\neq j$, and there is at least one letter in $W$ whose position in $E_1$ is shifted by one place from its position in $E_2$. Thus, the identical extensions $E_1$ and $E_2$ are not square-free, a contradiction.
\end{proof}

\begin{proof}[Proof of \cref{thm: emily_main}]
Let $W$ be an extremal square-free of length $n$ word on a $k$-letter alphabet. By \cref{lemma: distinct-extension}, every almost-square in $W$ can be extended to become a square factor in at most two ways: extending the left half to match the square-free right half and vice versa. There are $k(n+1)$ extensions of $W$, so there must be at least $\frac{k(n+1)}{2}$ distinct almost-squares in $W$. Using \cref{thm: almost-square}, we have
$$2n\log_{\frac{5}{4}}{n} \geq \frac{k(n+1)}{2},$$ so
$4\log_{\frac{5}{4}}{n} > k$. It follows that
$$n > \paren{\frac{5}{4}}^{\frac{k}{4}},$$ as desired.
\end{proof}

\section{Future directions}\label{section:future}

There are many patterns $P$ for which one can explore extremal $P$-avoiding words. One family of avoidable patterns that closely resembles the family studied in \cref{section: natasha} of this paper is the family of patterns of the form $XYXYX\dots XYX$. 

The bounds that we obtained in \cref{thm: almost-square} and \cref{thm: emily_main} are not known to be sharp and may be improved. Additionally, it is unknown whether an upper bound corresponding to the lower bound given in \cref{thm: emily_main} exists.

\begin{question}
Does there exist an upper bound on the longest possible length of an extremal square-free word on a $k$-letter alphabet?
\end{question}

Computer experiments conducted by Grytczuk et al.\ \cite{grytczuk2019extremal} failed to find an extremal square-free word over a four-letter alphabet with length less than 100.  
We conducted similar computer experiments for extremal abelian square-free words, where an abelian square is the concatenation of two nonempty words that are permutations of each other. 
We found that a shortest extremal abelian square-free word over the four-letter alphabet $\set{a,b,c,d}$ is 
$$abcdbacbdcba.$$
It is known that there are infinitely many abelian square-free words over a four-letter alphabet \cite{keranen1992abelian}, and our computer experiments lead us to conjecture that the same is true for extremal abelian square-free words over a four-letter alphabet. We believe that it may be possible to prove this by adapting the method that Grytczuk et al.\ used to prove their main result in \cite{grytczuk2019extremal}.
\begin{conjecture}
There are infinitely many extremal abelian square-free words over a four-letter alphabet.
\end{conjecture}

\appendix
\section{Proving the inequalities from the proof of \cref{lemma:almost-square}}

In this appendix, we prove the inequalities used in the proof of \cref{lemma:almost-square}.
Here, $WW'V = XX'$, where $WW'$ and $XX'$ are almost-squares and ${4\leq |V| \leq \frac{1}{2}|W|}$. 

\noindent\textbf{Claim 1.} $|W|<|X|<|WW'|$.
\begin{proof}[Proof of Claim 1.]
Using $|V|\geq 4$, we have $$|X| \geq \frac{|WW'V|-1}{2}\geq \frac{|WW'|+3}{2}\geq \frac{|W|+(|W|-1)+3}{2}>|W|.$$
Similarly, using $|V| \leq |W|$ and $|W'|\geq 3$, we have
\begin{align*}
    |X|\leq \frac{|WW'V|+1}{2} \leq\frac{|WW'| + |W|+(|W'|-3)+1}{2} < |WW'|. &\qedhere
\end{align*}
\end{proof}

\noindent\textbf{Claim 2.} $ |V|<|X| -1$.
\begin{proof}[Proof of Claim 2.]
Using $|W|\geq 2|V|\geq 8$, we have that
\begin{align*}
   &|W'|\geq |W| - 1\geq 7\\
   \Longrightarrow &|WW'|\geq |V| + 7\\
   \Longrightarrow &|XX'| = |WW'V| \geq 2|V|+7\\
   \Longrightarrow &|X| > |V|+1. & &\qedhere
\end{align*}
\end{proof}
\noindent\textbf{Claim 3.} $0 < \ell <|W|-1$, where $\ell := |V| - |X|+|W|$.
\begin{proof}[Proof of Claim 3.] Using $|V|\geq 4$, we have 
\begin{align*}
    \ell &= |V| - |X|+|W| \\
    &\geq |V| - \frac{|WW'V|+1}{2} + \frac{|WW'|-1}{2} \\
    &= \frac{|V|}{2} - 1 
    \geq 1. 
\end{align*}
The inequality $\ell <|W|-1$ follows directly from $|V|<|X| -1$ in Claim 2.
\end{proof}

\section*{Acknowledgments}
This research was conducted at the 2020 University of Minnesota Duluth Research Experience for Undergraduates (REU) program, which is supported by NSF-DMS grant 1949884 and NSA grant H98230-20-1-0009.  We would like to thank Joe Gallian for organizing the program, suggesting the problem, and supervising the research. We would also like to thank Amanda Burcroff, Colin Defant, Yelena Mandelshtam, and Joe Gallian for reading this paper and giving valuable suggestions and Amanda Burcroff for useful discussions about this research.

\bibliographystyle{plain}
\bibliography{ref}

\end{document}